\documentclass[11pt, twoside, letterpaper]{amsart}
\usepackage{amsmath, hyperref}

\newtheorem{thm}{Theorem}[section]

\newtheorem{lem}[thm]{Lemma}

\theoremstyle{definition}
\newtheorem{defn}[thm]{Definition}

\theoremstyle{remark}
\newtheorem{rem}[thm]{Remark}

\numberwithin{equation}{section}
\newcommand{\set}[1]{\left\{#1\right\}}
\newcommand{\Real}{\mathbb R}
\newcommand{\Natural}{\mathbb N}

\newcommand{\such}{\ | \ }
\newcommand{\nin}{n \in \Natural}
\newcommand{\kin}{k \in \Natural}
\newcommand{\prob}{\mathbb{P}}
\newcommand{\Exp}{\mathcal E}

\newcommand{\expec}{\mathbb{E}}
\newcommand{\expecp}{\expec_\prob}

\newcommand{\probtriple}{(\Omega, \mathcal{F}, \prob)}

\newcommand{\basisp}{(\Omega, \, \bF, \, \prob)}

\newcommand{\F}{\mathcal{F}}

\newcommand{\cadlag}{c\`adl\`ag}
\newcommand{\caglad}{c\`agl\`ad}
\newcommand{\ud}{\mathrm d}

\newcommand{\liminfn}{\liminf_{n \to \infty}}

\newcommand{\limn}{\lim_{n \to \infty}}

\newcommand{\zi}{{\Real_+}}

\newcommand{\zic}{[0,\infty]}

\newcommand{\zo}{[0, 1)}

\newcommand{\dya}{\mathbb{D}}

\newcommand{\on}{\, ; \,}

\newcommand{\cS}{{}^\mathsf{c} \kern-0.21em S}
\newcommand{\cH}{{}^\mathsf{c} \kern-0.23em H}
\newcommand{\cMM}{{}^\mathsf{c} \kern-0.19em [M, M]}

\newcommand{\pare}[1]{\left(#1\right)}
\newcommand{\bra}[1]{\left[#1\right]}
\newcommand{\dbra}[1]{[\kern-0.15em[ #1 ]\kern-0.15em]}
\newcommand{\dbraco}[1]{[\kern-0.15em[ #1 [\kern-0.15em[}
\newcommand{\dbraoc}[1]{]\kern-0.15em] #1 ]\kern-0.15em]}
\newcommand{\dbraoo}[1]{]\kern-0.15em] #1 [\kern-0.15em[}

\newcommand{\oL}{\overline{L}}

\newcommand{\bF}{\mathbf{F}}

\newcommand{\dfn}{\, := \,}

\newcommand{\indic}{\mathbb{I}}

\newcommand{\tir}{t \in \Real_+}
\newcommand{\sir}{s \in \Real_+}

\newcommand{\Lc}{\mathcal{L}_0}
\newcommand{\Mz}{\mathcal{M}_0}

\newcommand{\rl}{\rho_L}

\newcommand{\xix}{{(x, x^*)}}
\newcommand{\xixw}{{x, x^*}}

\begin{document}

\title{On the characterisation of honest times that avoid all stopping times}%
\author{Constantinos Kardaras}%
\address{Constantinos Kardaras, Statistics Department, London School of Economics and Political Science, 10 Houghton Street, London, WC2A 2AE, UK.}%
\email{k.kardaras@lse.ac.uk}%

\thanks{The author would like to thank two anonymous referees for helpful suggestions that greatly improved the presentation and content of the paper.}%

\date{\today}%

%----------------------------------------------------------------
\begin{abstract}
We present a short and self-contained proof of the following result: a random time is an honest time that avoids all stopping times if and only if it coincides with the (last) time of maximum of a nonnegative local martingale with zero terminal value and no jumps while at its running supremum, where the latter running supremum process is continuous. Illustrative examples involving local martingales with discontinuous paths are provided.
\end{abstract}

\maketitle

% ----------------------------------------------------------------

\section{The Characterisation Result} \label{sec: hon times avoid stop times}

\subsection{Honest times that avoid all stopping times}
Let $\basisp$ be a filtered probability space, where $\bF = (\F_t)_{t \in \zi}$ is a filtration satisfying the usual conditions of right-continuity and saturation by $\prob$-null sets of $\F \dfn \bigvee_{t \in \Real_+} \F_t$. All (local) martingales and supermartingales on $\basisp$ are assumed to have $\prob$-a.s. \cadlag \ paths.

\begin{defn} \label{defn: rand_times}
A \textsl{random time} is a $\zic$-valued, $\F$-measurable random variable. The random time $\rho$ is said to \textsl{avoid all stopping times} if $\prob[\rho = \tau] = 0$ holds whenever $\tau$ is a (possibly, infinite-valued) stopping time. The random time $\rho$ is called an \textsl{honest time} if for all $\tir$ there exists an $\F_t$-measurable random variable $R_t$ such that $\rho = R_t$ holds on $\set{\rho \leq t}$.
\end{defn}

Honest times constitute the most important class of random times outside the realm of stopping times. They have been extensively studied in the literature, especially in relation to filtration enlargements. It is impossible to present here the vast literature on the subject of honest times; we indicatively mention the early papers \cite{MR0326848}, \cite{MR509204} \cite{MR511775},  \cite{MR519998} and \cite{MR519996}, as well as the monographs \cite{MR604176} and \cite{MR884713}. Lately, there has been considerable revival to the study of honest times, due to questions arising from the field of Financial Mathematics---see, for example, \cite{MR1802597}, \cite{NP}, \cite{FJS} and the references therein.

\subsection{The class $\Mz$} \label{subsec: class_M}

Define $\Mz$ to be the class of all nonnegative local martingales $L$ such that $L_0 = 1$, the running supremum process $L^*  \dfn \sup_{t \in [0, \cdot]} L_t$ is continuous (up to a $\prob$-evanescent set), and $\prob \bra{L_\infty = 0} = 1$ holds, where $L_\infty \dfn \lim_{t \to \infty} L_t$. (Note that the limit in the definition of $L_\infty$  exists in the $\prob$-a.s. sense, in view of the nonnegative supermartingale convergence theorem.) 

For $L \in \Mz$, define\footnote{As usual, for any \cadlag \ process $X$, $X_{-}$ denotes the \caglad \ process defined in a way such that $X_{t-}$ is the left limit of $X$ at $t \in (0, \infty)$; by convention, we also set $X_{0-} = X_0$.}
\begin{equation} \label{eq: rl}
\rl \dfn \sup \set{\tir \such L_{t-} = L^*_{t-}},
\end{equation}
where note that $L_{0-} = 1 = L^*_{0-}$ implies that the (random) set $\set{\tir \such L_{t-} = L^*_{t-}}$ is non-empty. Since $\prob \bra{L_\infty = 0} = 1$ holds for $L \in \Mz$, it follows that $\prob \bra{\rl < \infty} = 1$.

For $L \in \Mz$ and $\tir$, define $R_t \dfn \sup \set{s \in [0, t] \such L_{s-} = L^*_{s-}} \wedge t$, which is an $\F_t$-measurable random variable such that $\rl = R_t$ holds on $\set{\rl \leq t}$. It follows that $\rl$ is an honest time whenever $L \in \Mz$.

\subsection{The class $\Lc$}
Let $L \in \Mz$. In view of \eqref{eq: rl}, $\rl$ coincides with the end of of the predictable set $\set{L_- = L_-^*}$. Using the $\prob$-a.s. left-continuity of $L_-$ and the $\prob$-a.s. continuity of $L^*$, as well as the definition of $\rl$ from \eqref{eq: rl}, we obtain that $L_{\rl -} = L^*_{\rl -} = L^*_{\rl}$ holds in the $\prob$-a.s. sense. (In particular, the ``$\sup$'' in \eqref{eq: rl} is really a ``$\max$''.) If one wishes to ensure that $\rl$ is an actual time of maximum of $L$, it suffices to ask that $L$ has no jumps when $L_-$ is at its running supremum. Motivated by this observation, we define the class $\Lc$ to consist of all $L \in \Mz$ with the additional property that $\set{L_- = L^*_-} \subseteq \set{\Delta L = 0}$ holds up to a $\prob$-evanescent set.\footnote{One could also ask that $\set{L_- = L^*_-} \subseteq \set{\Delta L > 0}$  holds up to a $\prob$-evanescent set; given that $L \in \Lc$, only downwards jumps are possible when $L_-$ is at its running supremum. Furthermore, note that a process $L \in \Lc$  never jumps (downwards) when $L_-$ is at its running supremum; however, it may jump (upwards) \emph{to} its running supremum. For a concrete example of such a case, see \S \ref{subsec: other_exa}.} Whenever $L \in \Lc$, it $\prob$-a.s. holds that $L_{\rl-} = L_{\rl} = L_{\rl}^*$; in fact, as Theorem \ref{thm: main} will imply, the previous random variables are also equal to $ L^*_{\infty}$, which makes $\rl$ a time of \emph{overall} maximum of $L \in \Lc$. On the other hand, if $L \in \Mz \setminus \Lc$ it may happen that $L$ does not achieve its overall supremum; furthermore, it may also happen that $\rl$ fails to avoid all stopping times---for both previous points, see Remark \ref{rem: M vs L}.

\subsection{The characterisation result}

The following result shows that, for $L \in \Lc$, the random time $\rl$ defined in \eqref{eq: rl} is the canonical example of an honest time that avoids all stopping times.

\begin{thm} \label{thm: main}
For a random time $\rho$, the following two statements are equivalent:
\begin{enumerate}
 \item $\rho$ is an honest time that avoids all stopping times.
 \item $\rho = \rl$ holds in the $\prob$-a.s. sense for some $L \in \Lc$.  
\end{enumerate}
Under (any of) the previous conditions, the equality $L_{\rho-} = L_\rho = L^*_\infty$ holds in the $\prob$-a.s. sense; furthermore, $\prob \bra{\rho > t \such \F_t} = L_t / L^*_t$ in the $\prob$-a.s. sense is valid for all $\tir$.
\end{thm}

We proceed with some remarks on Theorem \ref{thm: main}, the proof of which is given in Section \ref{sec: proof}. Section \ref{sec: examples} contains examples involving jump processes, illustrating Theorem \ref{thm: main}.

\begin{rem}
Along with $\rl$ from \eqref{eq: rl}, for $L \in \Mz$ define also $\rl' \dfn \sup \set{\tir \such L_t = L^*_t}$. (When $L \in \Lc$, it is straightforward to check that $\prob \bra{\rl \leq \rl'} = 1$.) %In fact, it will come as a consequence of Theorem \ref{thm: main} that $\prob \bra{\rl = \rl'} = 1$.  Recall the definition of $\rl'$ in \eqref{eq: rlp}.
Under the additional proviso that the filtration is continuous---meaning that all martingales on $\basisp$ have $\prob$-a.s. continuous paths---it was shown in \cite[Corollary 2.4 and Theorem 4.1]{MR2247846} that a random time $\rho$ is an honest time\footnote{Actually, both \cite{MR2247846} and \cite{NP} define an honest time as the end of an optional set, a concept that can be seen to be equivalent to the one in Definition \ref{defn: rand_times}; see, for example, \cite[Proposition 5.1]{MR604176}.} that avoids all stopping times if and only if $\rho = \rl'$ holds for some $L \in \Mz$. When the filtration is continuous, $\prob \bra{\rl = \rl'} = 1$ holds and the classes $\Mz$ and $\Lc$ coincide; therefore, the aforementioned results in \cite{MR2247846} constitute a special case of Theorem \ref{thm: main}. In \cite[Theorem 3.2]{NP}, it is shown that whenever $\rho$ is an honest time that avoids all stopping times, then $\rho = \rl'$ holds for some $L \in \Mz$. Although there is no assumption regarding filtration continuity in \cite{NP}, Theorem \ref{thm: main} is stronger since it gives a full characterisation; indeed, when the filtration fails to be continuous, Remark \ref{rem: M vs L} shows that $\Lc$ may be a strict subclass of $\Mz$. %; furthermore, the equality $\prob \bra{\rl = \rl'} = 1$ is actually a consequence of Theorem \ref{}
Note that the random time $\rl$ (for $L \in \Mz$) is considered in the present paper in place of $\rl'$ that was used in \cite{NP}, as it ties better with the definition of the important class $\Lc \subseteq \Mz$. Finally, it should be mentioned that the arguments in \cite{NP} make heavy use of previously-established results regarding so-called processes of class $(\Sigma)$; in contrast,  Section \ref{sec: proof} contains a relatively short and self-contained proof of Theorem \ref{thm: main}.
\end{rem}

\begin{rem} \label{rem: M vs L}
Consider a complete probability space $(\Omega, \, \F, \, \prob)$ that supports a $\Real_+$-valued random variable $\tau$ such that $\prob \bra{\tau > t} = e^{-t}$ holds for all $\tir$. Let $\bF$ be the usual augmentation of the smallest filtration which makes $\tau$ a stopping time. Define the process $L$ via $L_t = \exp(t) \indic_{\set{t < \tau}}$ for all $\tir$. It is straightforward to check that $L \in \Mz$, as well as $\rho_L = \tau$. In particular, $\rl$ fails the requirement to avoid all stopping times in a dramatic fashion, since it is actually equal to a stopping time. Note also that $L \notin \Lc$, since $\Delta L_{\rho_L} = - L_{\rho_L -} = - \exp(\tau) < 0$ holds.
\end{rem}

\begin{rem} \label{rem: unique}
Let $\rho$ be an honest time that avoids all stopping times. The process $L \in \Lc$ such that $\rho = \rl$, which exists in view of Theorem \ref{thm: main}, is necessarily unique (up a $\prob$-evanescent set). Indeed, let $M \in \Lc$ be another process such that $\prob \bra{\rho = \rho_M} = 1$. In view of Theorem \ref{thm: main}, one obtains the process equality $L / L^* = M / M^*$, up to a $\prob$-evanescent set. The integration-by-parts formula implies that
\begin{align*}
\frac{L}{L^*} &= 1 + \int_0^\cdot \frac{1}{L^*_t} \ud L_t -  \int_0^\cdot \frac{L_t}{(L^*_t)^2} \ud L^*_t   \\
&= 1 + \int_0^\cdot \frac{1}{L^*_t} \ud L_t -  \int_0^\cdot \frac{1}{L^*_t} \ud L^*_t   = 1 + \int_0^\cdot \frac{1}{L^*_t} \ud L_t - \log \pare{L^*_t},
\end{align*}
where the facts that $\int_0^\infty \indic_{\set{L_t < L^*_t}} \ud L^*_t = 0$ and $L^*$ has continuous paths were used in the previous equalities. The above calculation provides the Doob-Meyer decomposition of $L / L^*$; in exactly the same way, we obtain that $M / M^*= 1 + \int_0^\cdot \pare{1 / M^*_t} \ud M_t -  \log \pare{M^*_t}$ is the Doob-Meyer decomposition of $M / M^*$. Combining the equality $L / L^* = M / M^*$ with uniqueness of the Doob-Meyer decomposition, the process equalities $\log \pare{L^*} = \log \pare{M^*}$ and $\int_0^\cdot \pare{1 / L^*_t} \ud L_t = \int_0^\cdot \pare{1 / M^*_t} \ud M_t$ follow; from these, one concludes in a straightforward way that $L = M$.
\end{rem}

%\begin{rem}
%Whenever $L \in \Lc$ and $\rho$ is \emph{any} random time such that $\prob \bra{L_\rho = L^*_\infty} = 1$ holds, Lemma \ref{lem: unique} in Section \ref{sec: proof} implies that $\prob \bra{\rho = \rl} = 1$. Therefore, a process $L \in \Lc$ has a $\prob$-a.s. unique time of overall maximum. %In particular, $\prob \bra{\rml = \rmal < \infty} = 1$ holds whenever $L \in \Lc$.
%\end{rem}

\section{Examples Involving Processes with Jumps} \label{sec: examples}

In this section we present examples where the process $L \in \Lc$ corresponding (in view of Theorem \ref{thm: main}) to an honest time $\rho$ that avoids all stopping times has jumps. By Remark \ref{rem: unique}, the aforementioned correspondence is one-to-one; then, it follows that Theorem \ref{thm: main} has indeed a wider scope compared to the corresponding result that restricts filtrations to be continuous.

\subsection{Maximum of downwards drifting spectrally negative L\'evy processes with paths of infinite variation} \label{subsec: Levy_exa}

On the filtered probability space $\basisp$, assume that $X$ is a one-dimensional \cadlag \ L\'evy process with $X_0 = 0$. For information about L\'evy processes, the interested reader can check \cite{MR1739520}, a book which we shall be referring to in the following discussion. %An alternative reference is \cite{MR1406564}.

The probability law of the process $X$ can be fully characterised by its \textsl{L\'evy triplet} $(\alpha, \sigma^2, \nu)$, where $\alpha \in \Real$ equals the drift rate of the L\'evy process $X - \sum_{t \leq \cdot} \Delta X_t \indic_{\set{|\Delta X_t| > 1}}$, $\sigma^2 \in \Real_+$ is the diffusion coefficient, and $\nu$ is a L\'evy measure on $\Real \setminus \set{0}$ (equipped with its Borel sigma-field), which means that $\int_{\Real \setminus \set{0}} \pare{1 \wedge |x|^2} \nu [\ud x] < \infty$.

The first assumption on $X$ is that of no positive jumps; in terms of the L\'evy measure:
\begin{enumerate}
	\item[(L1)] $\nu \bra{(0, \infty)} = 0$.
\end{enumerate}
By \cite[Example 25.11]{MR1739520}, condition (L1) implies that $\expec \bra{\exp \pare{z X_t}} < \infty$ for $z \in \Real_+$. Therefore, one may consider the \textsl{Laplace exponent} function $\theta: \Real_+ \mapsto \Real$, defined implicitly via $\exp( t \theta(z)) = \expecp \bra{\exp \pare{z X_t}}$ for $z \in \Real_+$ and $t \in \Real_+$. By the L\'evy-Khintchine representation \cite[Section 8]{MR1739520},
\begin{equation} \label{eq: Levy-Khin}
\theta(z) = \alpha z + \frac{1}{2} \sigma^2 z^2 + \int_{(- \infty, 0)} \pare{\exp(z x) - 1 - z x \indic_{[-1,0)} (x)} \nu [\ud x], \quad \forall z \in \Real_+.
\end{equation}

The following is our second assumption on $X$:
\begin{enumerate}
	\item[(L2)] $\alpha + \int_{(- \infty, -1)} x \nu [\ud x] < 0$.
\end{enumerate}
Given (L1), condition (L2) is equivalent to asking that $\prob \bra{\lim_{t \to \infty} X_t = - \infty} = 1$, i.e., that $X$ is downwards drifting. To wit, first note that the function $\theta$ has a derivative $\theta'$ on $(0, \infty)$, and it is straightforward to see that $\theta'(0+) := \lim_{z \downarrow 0} \theta' (z) = \alpha + \int_{(- \infty, -1)} x \nu [\ud x] < 0$, the last strict inequality holding from condition (L2). A straightforward argument using the equality $\theta (z) = \log \pare{ \expec \bra{\exp \pare{z X_1}} }$ for all $z \in \Real_+$ shows that $\expec \bra{X_1} = \theta'(0+) < 0$, which immediately implies that $\prob \bra{\lim_{t \to \infty} X_t = - \infty} = 1$, in view of the law of large numbers.

Finally, we introduce the last assumption on $X$, equivalent to saying that the paths of $X$ are of infinite (first) variation:
\begin{enumerate}
	\item[(L3)] If $\sigma^2 = 0$, then $\int_{(-1, 0)} |x| \nu [\ud x] = \infty$.
\end{enumerate}

With $X^* \dfn \sup_{t \in [0, \cdot]} X_t$ denoting the running supremum process of $X$, define the random time
\begin{equation} \label{eq: rx}
\rho \dfn \sup \set{\tir \such X_{t-} = X^*_{t-}}.
\end{equation}
In what follows, we shall show that $\rho$ is an honest time that avoids all stopping times, by explicitly computing $L \in \Lc$ such that $\rho = \rl$.

Assume the validity of all conditions (L1), (L2) and (L3). The Laplace exponent function $\theta$ defined in \eqref{eq: Levy-Khin} is such that $\theta''(z) = \sigma^2 + \int_{(- \infty, 0)} x^2 \exp(z x) \nu [\ud x]$ for $z \in (0, \infty)$; in particular, it is convex. Furthermore, $\lim_{z \to \infty} \pare{\theta(z) / z^2} = \sigma^2 / 2$, while $\lim_{z \to \infty} \pare{\theta(z) / z} = \alpha - \int_{(-1, 0)} x \nu [\ud x] = \infty$ holds if $\sigma^2 = 0$ in view of condition (L3). It follows that $\lim_{z \to \infty} \theta(z) = \infty$.  The facts $\theta(0) = 0$, $\theta'(0+) < 0$ and $\lim_{z \to \infty} \theta(z) = \infty$, combined with the convexity of $\theta$, imply that there exists a unique $z_0 \in (0, \infty)$ such that $\theta(z_0) = 0$. A straightforward argument using the L\'evy property of $X$ and the definition of $\theta$ shows that the process $L \dfn \exp(z_0 X)$ is a martingale on $\basisp$ such that $L_0 = 1$, $L^*$ has continuous paths, and $\lim_{t \to \infty} L_t = 0$, all holding in the $\prob$-a.s. sense. It follows that $L \in \Mz$; furthermore, since $z_0 > 0$, a comparison of \eqref{eq: rl} and \eqref{eq: rx} shows that $\rho = \rho_L$.

In order to show that $L \in \Lc$, it remains to establish that the set $\set{L_- = L_-^*, \, \Delta L \neq 0}$ is $\prob$-evanescent. Since $L = \exp(z_0 X)$ with $z_0 > 0$, this last condition is equivalent to $\prob$-evanescence of $\set{X_- = X_-^*, \, \Delta X \neq 0}$, which is exactly the content of the following result.

\begin{lem}
Assume the validity of conditions \emph{(L1)}, \emph{(L2)} and \emph{(L3)}. Then, $\set{X_- = X_-^*, \, \Delta X \neq 0}$ is $\prob$-evanescent.
\end{lem}

\begin{proof}
Condition (L3) implies that $\liminf_{t \downarrow 0} \pare{X_t / t} = - \infty$; indeed, this follows from \cite[Theorem 47.1]{{MR1739520}}. (Look also at \cite[Definition 11.9]{{MR1739520}} for the concept of L\'evy processes of so-called ``type C.'') It follows that $\prob \bra{\inf_{s \in [0, t]} X_s = 0} = 0$ holds for all $t \in (0, \infty)$. Furthermore, in view of \cite[Remark 45.9]{{MR1739520}}, the probability laws of $\inf_{s \in [0, t]} X_s$ and $X_t - X^*_t$ are the same for any fixed $\tir$. Combining the previous, it follows that $\prob \bra{X_t = X^*_t} = \prob \bra{\inf_{s \in [0, t]} X_s = 0} = 0$ holds for all $t \in (0, \infty)$; in particular, $\int_{\Real_+} \prob \bra{X_{t-} = X^*_{t-}} \ud t = 0$. With $\mu$ denoting the jump measure of $X$, and since $\set{X_- = X^*_-}$ is a predictable set, a use of Fubini's theorem gives
\begin{align*}
\expec \bra{\int_{\Real_+ \times (- \infty, 0)} \indic_{\set{X_{t-} = X_{t-}^*}} \mu \bra{\ud t, \ud x}} &= \expec \bra{\int_{(- \infty, 0)} \pare{ \int_{\Real_+} \indic_{\set{X_{t-} = X_{t-}^*}} \ud t } \nu \bra{\ud x}} \\
&= \int_{(- \infty, 0)} \pare{ \int_{\Real_+}  \prob \bra{X_{t-} = X_{t-}^*} \ud t } \nu \bra{\ud x} = 0,
\end{align*}
which implies that $\int_{\Real_+ \times (- \infty, 0)} \indic_{\set{X_{t-} = X_{t-}^*}} \mu \bra{\ud t, \ud x} = 0$ holds in the $\prob$-a.s. sense. The latter is equivalent to that $\set{X_- = X_-^*, \, \Delta X < 0}$ is $\prob$-evanescent. Since $\Delta X \leq 0$, the proof is complete.
\end{proof}

Note that $L$ has $\prob$-a.s. continuous paths only in the case $\nu \equiv 0$; therefore, the above construction provides a plethora of examples of honest times that avoid all stopping times for which the unique (in view of Remark \ref{rem: unique}) representative $L \in \Lc$ with the property that $\rho = \rho_L$ has jumps.

\subsection{Geometric Brownian motion with jumps no higher than its running supremum} \label{subsec: other_exa}

Consider a probability space $\probtriple$, rich enough to support the following \emph{independent} elements:
\begin{itemize}
	\item a process $W = (W_t)_{\tir}$, which is a standard Brownian motion in its natural filtration;
	\item A sequence $(\sigma_n)_{\nin}$ of independent and identically distributed random variables having the exponential law with rate parameter $\lambda \in (0, \infty)$.
	\item A sequence $(U_n)_{\nin}$ of independent and identically distributed random variables having the standard uniform law on $[0,1]$.
\end{itemize}
Define $\tau_0 \dfn 0$ and $\tau_n = \sum_{m=1}^n \sigma_m$ for all $\nin$; then, the process $N$ defined via $N_t = \sum_{n=1}^\infty \indic_{\set{\tau_n \leq t}}$ for all $\tir$ is a Poisson process (in its own filtration) with arrival rate $\lambda$. Define also the compound Poisson process $C$ via $C_t = \sum_{n=1}^{N_t} U_n$ for all $\tir$. Let $\bF$ be the usual augmentation of the smallest filtration that makes $W$ and $C$ adapted. Note that $N$ is $\bF$-adapted, and that $W$ and $C$ are independent.

Given the above ingredients, we shall construct $L \in \Lc$ that behaves like an exponential Brownian motion with parameter $\sigma \in (0, \infty)$ in each stochastic interval $\dbraco{\tau_{n-1}, \tau_n}$ for all $\nin$, and then will jump at each time $\tau_n$ to a level that will be at most equal to $L^*_{\tau_n -}$. In contrast to \S \ref{subsec: Levy_exa}, $L$ here will be allowed to jump upwards; however, the arrival rate of jumps will be finite and equal to $\lambda$.

Define $\Xi \dfn \set{(x, x^*) \in (0, \infty)^2 \such x \leq x^*}$, corresponding to the spate space of a nonnegative local martingale and its running supremum. (We do not consider $x = 0$, since it is a ``cemetery'' state for nonnegative local martingales.) For each $\xix \in \Xi$, let $F(\cdot \on \xixw)$ be the cumulative distribution function of a probability law such that $F (y \on \xixw) = 0$ holds for $y \in (- \infty, -1)$ and $F (x^* / x - 1 \on \xixw) = 1$; in other words, the probability law corresponding to $F (\cdot \on \xixw)$ does not charge any set outside $[-1, (x^* - x) / x ]$. We also ask that
\begin{equation} \label{eq: zero_mean}
\int_{\Real} y \ud F (y \on \xixw) \equiv \int_{[-1, (x^* - x) / x ]} y \ud F (y \on \xixw) = 0, \quad \forall \xix \in \Xi. 
\end{equation}
This family will be used in the following manner: for each $\nin$, conditional on the pair $(L_{\tau_n -}, L^*_{\tau_n -})$ and when $L_{\tau_n -} > 0$, the ``relative jump'' $(L_{\tau_n} - L_{\tau_n -}) / L_{\tau_n -}$ of $L$ at time $\tau_n$ will have a probability law with cumulative distribution function $F(\cdot \on L_{\tau_n - }, L^*_{\tau_n -} )$. More details on the construction are given in the next paragraph. For the time being, note that there are many choices for the class of distributions $\set{F (\cdot \on \xixw) \such \xix \in \Xi}$ satisfying the aforementioned constraints. Possibly the simplest such class is the following: for $\xix \in \Xi$, $F (\cdot \on \xixw)$ corresponds to the probability law of a two-point-mass with probability $x/x^*$ equalling $(x^* - x )/ x$ and probability $1 - x / x^*$ equalling $-1$. According to the heuristic description given above, this particular choice corresponds to $L$ jumping at each time point $\tau_n$ (and if $L_{\tau_n -} > 0$) either to its running supremum $L_{\tau_n -}^*$ with probability $L_{\tau_n -} / L^*_{\tau_n -}$ or to zero (and then staying there forever) with probability $1 - L_{\tau_n -} / L^*_{\tau_n -}$.

We now proceed to the formal inductive construction of $L$. Let $L_0 = 1$, and assume that $L$ has been defined on the stochastic interval $\dbra{0, \tau_{n-1}}$ for some $\nin$. %, and that it is a nonnegative local martingale with continuous supremum process.
If $L_{\tau_{n-1}} = 0$, define $L_t = 0$ for all $t \in (\tau_n, \infty)$ and terminate the process. If $L_{\tau_{n-1}} > 0$, first define $L$ on $\dbraco{\tau_{n-1}, \tau_n}$ via
\begin{equation} \label{eq: L_induct}
L_t = L_{\tau_{n-1}} \exp \pare{\sigma (W_t - W_{\tau_{n-1}}) - \frac{\sigma^2}{2} (t - \tau_{n-1}) }, \quad \text{for } t \in (\tau_{n-1}, \tau_n).
\end{equation}
For $\xix \in \Xi$, let $F^{-1} (\cdot \on \xixw) : [0,1] \mapsto [-1, \infty)$ denote the ``inverse'' of $F(\cdot; \xixw)$, formally defined via
\[
F^{-1} (u \on \xixw) \dfn \sup \set{y \in \Real \such F(y \on \xixw) < u}, \quad \forall u \in [0,1],
\]
and assume that the mapping $F^{-1} : [0, 1] \times \Xi \mapsto [-1, \infty)$ is (jointly) Borel-measurable. Then, according to \cite[Theorem 1.2.2]{MR2722836}, the random variable $F^{-1} (U_n \on \xixw)$ has a law with cumulative distribution function $F (\cdot \on \xixw)$. With $j_n \dfn F^{-1} \pare{U_n \on L_{\tau_n - }, L^*_{\tau_n -}}$, set $L_{\tau_n} = L_{\tau_n -} (1 + j_n)$, which defines $L$ on the whole stochastic interval $\dbra{0, \tau_{n}}$ and completes the induction step. Since $\limn \tau_n = \infty$ holds in the $\prob$-a.s. sense, it follows that $L$ is defined for all times in $\Real_+$.

We proceed in showing that $L \in \Mz$. Note that $(L_{\tau_n -}, L^*_{\tau_n -})$ is independent of $U_n$ for all $\nin$; therefore, given $(L_{\tau_n -}, L^*_{\tau_n -})$ and $L_{\tau_n -} > 0$, $j_n$ has a law with cumulative distribution function $F \pare{\cdot \on L_{\tau_n - }, L^*_{\tau_n -}}$. In particular, we have the $\prob$-a.s. inequalities $0 \leq L_{\tau_n} \leq L_{\tau_n -} (L^*_{\tau_n -} / L_{\tau_n -}) = L^*_{\tau_n -}$, which imply that $L$ stays nonnegative and does not jump over its running supremum. This shows that $L^*$ is continuous in the $\prob$-a.s. sense. Furthermore, $L$ is a nonnegative semimartingale such that $L = \Exp(\sigma W + J)$ holds, where ``$\Exp$'' denotes the stochastic exponential operator and the the pure-jump process $J$ with $\Delta J \geq - 1$ is defined via $J_t = \sum_{n=1}^{N_t} j_n$ for all $\tir$. With $\eta$ denoting the predictable compensator of the jump measure of $J$, it is straightforward to check that $\eta \bra{\ud t, \ud y} = \lambda \ud t \ud F \pare{y \on L_{t - }, L^*_{t -}}$ holds for $(t, y) \in \Real_+ \times \Real$. In particular, since Fubini's theorem and \eqref{eq: zero_mean} imply that
\[
\int_{[0, \cdot] \times \Real} y \eta \bra{\ud t, \ud y}  = \int_{[0, \cdot]} \pare{ \int_{\Real} y \ud F \pare{y \on L_{t - }, L^*_{t -}}  } \lambda \ud t = 0
\]
identically holds, it follows that $J$ is a purely discontinuous local martingale. Since $\sigma W$ is a continuous local martingale, $L = \Exp(\sigma W + J) = \Exp(\sigma W) \Exp(J)$ is a nonnegative local martingale. The law of large numbers for Brownian motion and the fact that $\sigma \in (0, \infty)$ give the limiting equality $\Exp(\sigma W)_\infty = \exp \pare{ \lim_{t \to \infty} \pare{\sigma W_t - \sigma^2 t / 2}} = 0$, valid in the $\prob$-a.s. sense. Since $\prob \bra{\Exp(J)_\infty \in \Real_+} = 1$ holds in view of the nonnegative supermartingale convergence theorem, $L_\infty = \Exp(\sigma W)_\infty \Exp(J)_\infty = 0$ holds in the $\prob$-a.s. sense, which implies that $L \in \Mz$.

In order to establish that $L \in \Lc$, which will finalise the discussion of this example, it remains to show that $\set{L_- = L_-^*, \, \Delta L \neq 0}$ is $\prob$-evanescent. For each $\nin$, note that the random variable $\sigma_n \dfn \tau_n - \tau_{n-1}$ is independent of the sigma-field generated by $\F_{\tau_{n-1}}$ and (the whole process) $W$. Furthermore, there is zero probability that an exponential Brownian motion sampled at an independent random time is equal to either its running maximum or to any fixed value. The last two facts and \eqref{eq: L_induct} imply that $\prob \bra{L_{\tau_n -} = L_{\tau_n -}^* \such \F_{\tau_{n-1}}} = 0$ holds for all $\nin$. In view of the obvious set-inclusion $\set{\Delta L \neq 0} \subseteq \bigcup_{\nin} \dbra{\tau_n, \tau_n}$, valid up to a $\prob$-evanescent set, we deduce that $\set{L_- = L_-^*, \, \Delta L \neq 0}$ is $\prob$-evanescent.

\section{Proof of Theorem \ref{thm: main}} \label{sec: proof}

During the course of the proof of Theorem \ref{thm: main}, and in an effort to be as self-contained as possible, we shall provide full details for every step.

For a random time $\sigma$ and a process $X = (X_t)_{\tir}$, $X^\sigma = (X_{\sigma \wedge t})_{\tir}$ will denote throughout the process $X$ stopped at $\sigma$. For any unexplained, but fairly standard, notation and facts regarding stochastic analysis, we refer the reader to \cite{MR1780932}.

\subsection{Doob's maximal identity} \label{subsec: aux_res}

We start by proving a slightly elaborate version of Doob's maximal identity---see \cite{MR2247846}. It will be quite useful throughout, sometimes in its ``conditional'' version. 

\begin{lem} \label{lem: Doob_maximal}
Let $L$ be a nonnegative local martingale with $L_0 = 1$. Then, $\prob \bra{ L^*_\infty > x} \leq 1/x$ holds for all $x \in (1, \infty)$. Furthermore, $\prob \bra{ L^*_\infty > x} = 1/x$ holds for all $x \in (1, \infty)$ if and only if $L \in \Mz$.
\end{lem}

\begin{proof}
For $x \in (1, \infty)$, define the stopping time $\tau_x \dfn \inf \set{t \in \Real_+ \such L_t > x}$, and note that $\set{L^*_\infty > x} = \set{\tau_x < \infty}$. Since $\expec \bra{L^*_{\tau_x}} \leq x + \expec \bra{L_{\tau_x}} \leq x+1$, $L^{\tau_x}$ is a uniformly integrable martingale for all $x \in (1, \infty)$. It follows that $x \prob \bra{ L^*_\infty > x} = x \prob [\tau_x < \infty] = \expec [x \indic_{\set{\tau_x < \infty}}] \leq \expec [L_{\tau_x}] = 1$ for $x \in (1, \infty)$, with equality holding if and only if $\prob [L_{\tau_x} = x \indic_{\set{\tau_x < \infty}}] = 1$. Whenever $L \in \Mz$, the equality $\prob [L_{\tau_x} = x \indic_{\set{\tau_x < \infty}}] = 1$ is immediate for all $x \in (1, \infty)$. Conversely, assume that $\prob [L_{\tau_x} = x \indic_{\set{\tau_x < \infty}}] = 1$ holds for all for $x \in (1, \infty)$. It is clear that $L^*$ must have $\prob$-a.s. continuous paths; furthermore, since $\prob \big[ \bigcup_{\nin} \set{\tau_n = \infty} \big] = 1$, $\prob [ L_\infty = 0] = 1$ follows. Therefore, $L \in \Mz$.
\end{proof}

Suppose that the equivalence between conditions (1) and (2) of Theorem \ref{thm: main} has been established. For fixed $\tir$, let $\oL_t \dfn \sup_{v \in [t, \infty)} L_v$; the set-inclusions $\set{\oL_t > L^*_t} \subseteq \set{\rl > t} \subseteq  \set{ \oL_t \geq L^*_t}$ and a conditional version of Lemma \ref{lem: Doob_maximal} give
\[
\frac{L_t}{L^*_t} = \prob \bra{\oL_t > L^*_t \such \F_t} \leq \prob \bra{\rl > t \such \F_t}  \leq \prob \bra{\oL_t \geq L^*_t \such \F_t} = \frac{L_t}{L^*_t} , \quad \forall \tir.
\]
Since $\prob \bra{\rho = \rl} = 1$, it follows that $\prob \bra{\rho > t \such \F_t} =  L_t / L^*_t$ holds for all $\tir$.

Implication $(2) \Rightarrow (1)$ of Theorem \ref{thm: main} is dealt with in \S \ref{subsec: proof of 2 -> 1}. The more difficult implication $(1) \Rightarrow (2)$ is the content of \S \ref{subsec: proof of 1 -> 2}; there, the fact that $L_{\rho-} = L_\rho = L^*_\infty$ holds in the $\prob$-a.s.  sense is also established (in Lemma \ref{lem: max_rho}).

\subsection{Proof of implication $(2) \Rightarrow (1)$}  \label{subsec: proof of 2 -> 1}
It has already been shown in \S \ref{subsec: class_M} that $\rl$ is an honest time if $L \in \Mz$; in particular, $\rl$ is an honest time if $L \in \Lc$. Implication $(2) \Rightarrow (1)$ will follow once we establish that $\rl$ avoids all stopping times whenever $L \in \Lc$. To this end, fix some stopping time $\tau$; it will be shown below that $\prob[\rl = \tau \such \F_\tau] = 0$ holds up to a $\prob$-null set. Since $\prob \bra{\rl = \infty} = 0$, $\prob[\rl = \tau \such \F_\tau] = 0$ trivially holds (up to a $\prob$-null set) on $\set{\tau = \infty}$. Furthermore, note that $\set{\rl = \tau < \infty, \, L_\tau < L^*_{\tau}} \subseteq \set{\tau < \infty, \, L_{\tau-} = L^*_{\tau -}, \, \Delta L_{\tau} < 0}$; since $L \in \Lc$, the latter event has zero probability, from which we obtain that $\prob[\rl = \tau \such \F_\tau] = 0$ also holds on $\set{\tau < \infty, \, L_\tau < L^*_{\tau}}$, up to a $\prob$-null set. Finally, on $\set{\tau<\infty, \, L_\tau = L^*_{\tau}}$, where in particular $L_\tau > 0$, a conditional form of Lemma \ref{lem: Doob_maximal} gives that $\prob \big[ \sup_{t \in [\tau, \infty)} L_t > L^*_\tau \such \F_\tau \big] = L_\tau / L^*_\tau = 1$ holds; therefore, $\prob[\rl = \tau \such \F_\tau] = 0$ also holds on $\set{\tau<\infty, \, L_\tau = L^*_{\tau}}$, up to a $\prob$-null set.

\subsection{Proof of implication $(1) \Rightarrow (2)$ and the equality $L_{\rho-} = L_\rho = L^*_\infty$} \label{subsec: proof of 1 -> 2}
Throughout \S \ref{subsec: proof of 1 -> 2}, \emph{fix an honest time $\rho$ that avoids all stopping times}. Let $Z$ be the the $[0,1]$-valued (\cadlag)  Az\'ema supermartingale that satisfies $Z_t = \prob[\rho > t \such \F_t]$ for all $t \in \Real_+$. The next result follows from \cite[Lemma 4.3(i) and Proposition 5.1]{MR604176}---we provide its proof for completeness.

\begin{lem} \label{lem: Z is one}
With the above notation, $\prob [ Z_\rho = 1] = 1$ holds.
\end{lem}

\begin{proof}
Let $(R^0_t)_{\tir}$ be an adapted process such that $\rho = R^0_t$ holds on $\set{\rho \leq t}$ for all $\tir$. Note that the adapted process $(R^0_t \wedge t)_{\tir}$ has the same property as well; therefore, we may assume that $R^0_t \leq t$ holds for all $\tir$. With $\dya$ denoting a dense countable subset of $\Real_+$, define the process $R \dfn \lim_{\dya \ni t \downarrow \cdot} \big( \sup_{s \in \dya \cap (0, t)} R^0_s \big)$; then, $R$ is right-continuous, adapted and non-decreasing, and $R_t \leq t$ still holds for all $\tir$. Furthermore, since for $\sir$ and $\tir$ with $s \leq t$, $\rho = R^0_s = R^0_{t}$ holds on $\set{\rho \leq s} \subseteq \set{\rho \leq t}$, it follows that $\rho = R_t$ holds on $\set{\rho \leq t}$ for all $\tir$.  Define a $\set{0, 1}$-valued optional process $I$ via $I_t = \indic_{\set{R_t = t}}$ for $t \in \Real_+$. The properties of $R$ can be seen to imply $\set{I =1} \subseteq \dbra{0 ,\rho}$, as well as $I_\rho = 1$ on $\set{\rho < \infty}$; since $\prob[\rho = \infty] = 0$ holds due to the fact that $\rho$ avoids all stopping times, we conclude that $\prob[I_\rho = 1] = 1$. Fix a finite stopping time $\tau$. Using again the fact that $\rho$ avoids all stopping times, $Z_\tau = \prob[\rho \geq \tau \such \F_\tau]$ holds. Then, $I_\tau \in \F_\tau$ and $\set{I = 1} \subseteq \dbra{0, \rho}$ imply that $\expec \bra{I_\tau Z_\tau} = \expec \bra{I_\tau \indic_{\set{\tau \leq \rho}}} = \expec \bra{I_\tau}$. Since $I$ is $\set{0,1}$-valued and $Z$ is $[0,1]$-valued, $\expec \bra{I_\tau Z_\tau} = \expec \bra{I_\tau}$ implies that $\set{I_\tau = 1} \subseteq \set{Z_\tau = 1}$. Since the latter holds for all finite stopping times $\tau$ and both $I$ and $Z$ are optional, the optional section theorem implies that $\set{I = 1} \subseteq \set{Z = 1}$, modulo $\prob$-evanescence. Then, $\prob \bra{I_\rho = 1} = 1$ implies $\prob \bra{Z_\rho = 1} = 1$. 
\end{proof}

Continuing, let $A$ be the unique (up to $\prob$-evanescence) adapted, \cadlag, nonnegative and nondecreasing process such that $\expec [V_\rho] = \expec \bra{ \int_0^\infty V_t \ud A_t}$ holds for all nonnegative optional processes $V$---in other words, $A$ is the dual optional projection of $\indic_{\dbraco{\rho, \infty}}$. Since $\expec \bra{A_\tau - A_{\tau -}} = \prob[\rho = \tau] = 0$ holds for all finite stopping times $\tau$, the optional section theorem implies that $A_0 = 0$ and $A$ has $\prob$-a.s. continuous paths. Define also $M$ as the nonnegative uniformly integrable martingale such that $M_t = \expec \bra{A_\infty \such \F_t}$ holds for all $t \in \zi$. By the definition of $A$ and $M$, note that
\[
M_t = A_t + \expec \bra{A_\infty - A_{t} \such \F_t} = A_t + \prob \bra{\rho > t \such \F_t} = A_t + Z_t, \quad \forall \tir.
\]
Given the $\prob$-a.s. continuity of the paths of $A$, it follows that $Z = M - A$ is the (additive) Doob-Meyer decomposition of $Z$. The following result provides the \emph{multiplicative} Doob-Meyer decomposition of $Z$, a topic first treated in \cite{MR0184282}. In the present case where it is known that the predictable process $A$ is actually continuous, the proof simplifies.

\begin{lem} \label{lem: mult_decomp}
With the above notation, one has $Z = L (1 - K)$, where $L$ is a nonnegative local martingale with $L_0 = 1$ and $K$ is a $[0,1]$-valued nondecreasing adapted process with $\prob$-a.s. continuous paths. Furthermore, $A = \int_0^\cdot L_t \ud K_t$ holds.
\end{lem}

\begin{proof}
For each $\nin$, define the stopping time $\zeta_n \dfn \inf \set{\tir \such Z_t < 1/n}$. Furthermore, set $\zeta \dfn \lim_{n \to \infty} \zeta_n = \inf \set{\tir \such Z_{t-} = 0 \text{ or } Z_t = 0}$.

Define $K \dfn 1 - \exp \big( - \int_0^{\zeta \wedge \cdot} (1 / Z_t) \ud A_t \big)$, which obviously is a $[0,1]$-valued nondecreasing adapted process. The fact that $A$ has $\prob$-a.s. continuous paths implies that $K$ is $\prob$-a.s. continuous on $\dbra{0, \zeta_n}$ and that $\prob \bra{K_{\zeta_n} < 1} = 1$ holds for all $\nin$. Furthermore, it is straightforward to check that $A = A^\zeta$ holds; therefore, we conclude that $K$ has $\prob$-a.s. continuous paths.

Setting $L^n \dfn Z^{\zeta_n} / (1 - K^{\zeta_n})$, a straightforward application of the integration-by-parts formula gives $L^n = 1 + \int_0^{\zeta_n \wedge \cdot} (L^n_t / Z_t) \ud M_t$, implying that $L^n$ is a nonnegative local martingale for all $\nin$. For $m \leq n$, it holds that $L^m = L^n$ on $\dbra{0, \zeta_m}$; then, the nonnegative martingale convergence theorem implies that $\ell \dfn \lim_{n \to \infty} L^n_{\zeta_n}$ exists and is $\Real_+$-valued in the $\prob$-a.s. sense. One may therefore define a nonnegative \cadlag \ process $L$ such that $L = L^n$ holds on $\dbra{0, \zeta_n}$ for all $\nin$ and $L_t = \ell$ holds for all $t \geq \zeta$. In view of Lemma \ref{lem: Doob_maximal}, the fact that $L^{\zeta_n}$ is a nonnegative martingale with $L^{\zeta_n}_0 = 1$ implies that $\prob [ L^*_{\zeta_n} > x ] \leq 1/x$ holds for all $\nin$. Since $L = L^\zeta$ and $\prob \bra{\limn \zeta_n = \zeta } = 1$, we obtain that $\prob \bra{L^*_\infty < \infty} = 1$. Therefore, defining the stopping time $\tau_k \dfn \inf \set{\tir \such L_t > k}$ for all $k \in \Natural$, it follows that $\prob \bra{\lim_{k \to \infty} \tau_k = \infty} = 1$. Furthermore, since $L = L^\zeta$, $\prob \bra{\limn \zeta_n = \zeta } = 1$, and $\expec \bra{L_{\tau_k \wedge \zeta_n}} = \expec \big[ L^{\zeta_n}_{\tau_k} \big] = 1$ holds for all $\kin$ and $\nin$, Fatou's lemma gives
\[
\expec \bra{L^*_{\tau_k}} = \expec \bra{\limn L^*_{\tau_k \wedge \zeta_n}} \leq \liminfn \expec \bra{L^*_{\tau_k \wedge \zeta_n}} \leq \liminfn \pare{k + \expec \bra{L_{\tau_k \wedge \zeta_n}} } = k+1 < \infty, \quad \forall \kin.
\]
For $0 \leq s \leq t < \infty$, the (conditional version of the) dominated convergence theorem gives
\[
\expec \bra{L^{\tau_k}_t \such \F_s} = \expec \bra{ \limn L_{\tau_k \wedge \zeta_n \wedge t} \such \F_s} = \limn \expec \bra{  L^{\zeta_n}_{\tau_k \wedge t} \such \F_s} = \limn L^{\zeta_n}_{\tau_k \wedge s} = L^{\tau_k}_s, \quad \forall \kin.
\]
It follows that $L^{\tau_k}$ is a martingale for all $\kin$; therefore, $L$ is a nonnegative local martingale.

Since $K = K^\zeta$, $L = L^{\zeta}$ and $Z = Z^{\zeta}$, we conclude that $Z = L (1 - K)$ holds. By the integration-by-parts formula, $Z = 1 + \int_0^\cdot (1 - K_t) \ud L_t - \int_0^\cdot L_t \ud K_t$ holds; comparing with the Doob-Meyer decomposition $Z = M - A$ of $Z$, and recalling that $A_0 = 0$, we obtain that $A = \int_0^\cdot L_t \ud K_t$.
\end{proof}

\begin{lem} \label{lem: K uniform}
With the above notation, $K_\rho$ has the standard uniform law.
\end{lem}

\begin{proof}
For $u \in \zo$, define the stopping time $\tau_u \dfn \inf \set{t \in \zi \such K_t > u}$, with the convention $\tau_u = \infty$ if the last set is empty. Since $K$ has $\prob$-a.s. continuous paths, $K_{\tau_u} = u$ holds $\prob$-a.s. on $\set{\tau_u < \infty}$ for all $u \in \zo$. Recalling that $A = \int_0^\cdot L_t \ud K_t$ holds from Lemma \ref{lem: mult_decomp}, a use of the change-of-time technique gives
\begin{equation} \label{eq: time change}
\int_0^\infty f(K_t) \ud A_t = \int_0^\infty f(K_t) L_t \ud K_t =  \int_0^1 L_{\tau_u} \indic_{\set{\tau_u < \infty}} f(u) \ud u, \quad \text{for any Borel } f: \zo \mapsto \Real_+.
\end{equation}
Since $Z = L (1 - K)$, the facts that $Z \leq 1$ and $K \leq u$ hold up to $\prob$-evanescence on $\dbra{0, \tau_u}$ imply that $\prob \bra{L^*_{\tau_u} \leq 1/(1 - u)} = 1$ holds for all $u \in \zo$. Therefore, $\expec[L_{\tau_u}] = 1$ holds for all $u \in \zo$. Since $\prob [\rho = \infty] = 0$, it follows that $\prob[Z_\infty = 0] = 1$; then, $\prob[Z_\infty = L_\infty (1 - K_\infty)] = 1$ implies $\prob \bra{K_\infty < 1, \, L_\infty > 0} = 0$. Therefore, for $u \in \zo$, the set-inclusion $\set{\tau_u = \infty} \subseteq \set{K_\infty < 1}$ implies $\prob \bra{L_{\tau_u} \indic_{\set{\tau_u < \infty}}= L_{\tau_u}} = 1$. Then, $\expec[L_{\tau_u}] = 1$ gives $\expec \bra{ L_{\tau_u} \indic_{\set{\tau_u < \infty}}} = 1$  for $u \in \zo$. By Fubini's Theorem and \eqref{eq: time change}, we obtain $\expec \bra{f(K_\rho)} = \expec \bra{\int_0^\infty f(K_t) \ud A_t} = \int_0^1 f(u) \ud u$. Since the latter holds for any Borel $f: \zo \mapsto \Real_+$, it follows that $K_\rho$ has the standard uniform law.
\end{proof}

\begin{lem} \label{lem: max_rho}
With the above notation, it holds that $L \in \Mz$ and $\prob \big[L_{\rho -} = L_\rho = L^*_\infty \big] = 1$.
\end{lem}

\begin{proof}
Since $\prob[Z_\rho = L_\rho (1 - K_\rho)] = 1$, Lemma \ref{lem: Z is one} gives $\prob \bra{L_\rho = 1 / (1 - K_\rho)} = 1$. Then, Lemma \ref{lem: K uniform} implies that $\prob[L_\rho > x] = \prob[K_\rho > 1 - 1/x] = 1/x$ for all $x \in (1, \infty)$. As $\prob \bra{L_\rho \leq L^*_\infty} = 1$,  Lemma \ref{lem: Doob_maximal} implies both that $L \in \Mz$ and that $\prob \bra{L_\rho = L^*_\infty} = 1$. It remains to show that $\prob \bra{L_{\rho-} = L_{\rho}} = 1$, which is equivalent to $\expec \bra{|\Delta L_\rho|} = 0$. By the definition of $A$, it holds that $\expec \bra{|\Delta L_\rho|} = \expec \big[ \int_{\Real_+} |\Delta L_t| \ud A_t \big] = 0$, the last equality holding from the fact that $A$ is such that $A_0 = 0$ and has $\prob$-a.s. continuous paths (since $\rho$ avoids all stopping times), combined with the $\prob$-a.s. countability of the (random) set $\set{\tir \such \Delta L_t \neq 0}$. 
\end{proof}

\begin{lem} \label{lem: pre_final}
With the above notation, $\prob \bra{ \rho = \rl } = 1$ holds.
\end{lem}

\begin{proof}
Since $L_{\rho -} \leq L^*_{\rho -} \leq L^*_\infty$, the equality $\prob \big[L_{\rho -} = L^*_\infty \big] = 1$ that was established in Lemma \ref{lem: max_rho} implies that $\prob \big[L_{\rho-} = L^*_{\rho-} \big] = 1$; by the definition of $\rl$ in \eqref{eq: rl}, $\prob \bra{\rho \leq \rl} = 1$ is evident. For $\tir$, let $\oL_t \dfn \sup_{v \in [t, \infty)} L_v$ and note the set-inclusions $\set{\oL_t > L^*_t} \subseteq \set{\rho > t}$ and $\set{\rl > t} \subseteq \set{ \oL_t \geq L^*_t}$, valid modulo $\prob$. A use of the conditional version of Lemma \ref{lem: Doob_maximal} gives $\prob \bra{ \oL_t \geq L^*_t \such \F_t } =  L_t / L^*_t = \prob \bra{ \oL_t > L^*_t \such \F_t }$, for all $\tir$. It follows that $\prob \bra{\rl > t} \leq \prob \bra{\rho > t}$ holds for all $\tir$. Combined with $\prob \big[ \rho \leq \rl] = 1$, we obtain $\prob \big[ \rho = \rl \big] = 1$.
\end{proof}

The next result concludes the proof of implication $(1) \Rightarrow (2)$ of Theorem \ref{thm: main}.

\begin{lem} \label{lem: final}
With the above notation, it holds that $L \in \Lc$.
\end{lem}

\begin{proof}
A use of Lemma \ref{lem: max_rho} gives $L \in \Mz$ and $\prob \bra{\Delta L_\rho \neq 0} = 0$. If the set $\set{L_- = L_-^*, \, \Delta L \neq 0}$ failed to be $\prob$-evanescent, one would infer the existence of a stopping time $\tau$ with the property that $\prob \bra{\tau < \infty, \, L_{\tau -} = L_{\tau -}^*, \, \Delta L_{\tau} < 0} = \prob \bra{\tau < \infty} > 0$ holds. Recalling that $\prob \bra{\rho = \rl} = 1$ from Lemma \ref{lem: pre_final}, a conditional version of Lemma \ref{lem: Doob_maximal} gives
\[
\prob \bra{\rho = \tau \such \F_\tau} = \prob \bra{\rl = \tau \such \F_\tau} = 1 - \frac{L_{\tau}}{L^*_{\tau}} = 1 - \frac{L_{\tau -} + \Delta L_{\tau}}{L_{\tau -}}  = - \frac{\Delta L_{\tau}}{L_{\tau -}}.
\]
It follows that $\prob \bra{\rho = \tau \such \F_\tau} > 0$ holds on the $\F_\tau$-measurable event $\set{\tau < \infty, \, L_{\tau -} = L_{\tau -}^*, \, \Delta L_{\tau} < 0}$, implying that $\prob \bra{\Delta L_{\rho} < 0 } \geq \prob \bra{\Delta L_{\tau} < 0, \, \rho = \tau} > 0$, which is a contradiction. We deduce that $\set{L_- = L_-^*, \, \Delta L \neq 0}$ is $\prob$-evanescent, i.e., that $L \in \Lc$.
\end{proof}

\bibliographystyle{alpha}
\bibliography{hon_times_spa}
\end{document}